\documentclass[a4paper,12pt]{smfart}

\usepackage{ucs}
\usepackage[latin1]{inputenc}

\usepackage[francais,english]{babel} \AutoSpaceBeforeFDP
\usepackage[T1]{fontenc}
\usepackage{amstext,amssymb,amsthm}
\usepackage{smfthm}
\usepackage{graphicx}
\usepackage{epsfig}
\usepackage[all]{xy}
\usepackage{enumerate}
\usepackage{version}

\usepackage{marvosym}

\usepackage{geometry}
\usepackage{color}
\definecolor{darkred}{rgb}{0.9,0.,.2}
\definecolor{darkblue}{rgb}{0.,0.,.6}
\definecolor{darkgreen}{rgb}{0.,.6,0.1}
\usepackage[colorlinks=true,urlcolor=darkblue,linkcolor=darkred,citecolor=darkgreen]{hyperref}
\usepackage{subfig}

\newcommand{\N}{\mathbb{N}}

\newcommand{\R}{\mathbb{R}}
\newcommand{\C}{\mathcal{C}} 

\renewcommand{\ss}{\mathrm{SL_{n+1}(\mathbb{R})}}
\newcommand{\sss}{\mathrm{SL}}

\newcommand{\G}{\Gamma}
\newcommand{\g}{\gamma}

\renewcommand{\C}{\mathcal{C}}
\newcommand{\U}{\mathcal{U}}

\newcommand{\V}{\mathcal{V}}

\newcommand{\E}{\mathcal{E}}

\renewcommand{\AA}{\mathcal{A}}

\renewcommand{\O}{\Omega}

\renewcommand{\d}{d_{\Omega}}

\newcommand{\PP}{\mathbb{P}}

\newcommand{\Quo}{\Omega/\!\raisebox{-.90ex}{\ensuremath{\Gamma}}}
\newcommand*{\Quotient}[2]{\ensuremath{#1/\!\raisebox{-.90ex}{\ensuremath{#2}}}}

\newcommand{\Aut}{\textrm{Aut}}

\newcommand{\Stab}{\textrm{Stab}}

\newcommand{\Isom}{\textrm{Isom}}

\theoremstyle{plain}

\newtheorem{theorem}[section]{Théorème}

\theoremstyle{definition}

\theoremstyle{remark}

\newtheorem*{rem}{Remarque}

\title{Un lemme de Kazhdan-Margulis-Zassenhaus pour les géométries de Hilbert}

\author{
\href{mailto:crampon@math.u-strasbg.fr}{Mickaël Crampon}
}

\author{
\href{ludovic.marquis@univ-rennes1.fr}{Ludovic Marquis}
}

\date{} 

\address{\newline{}}
\address{\newline{}}

\email{crampon@math.u-strasbg.fr \\ ludovic.marquis@univ-rennes1.fr \newline{} }

\address{\newline{}}

\urladdr{
http://mikl.crampon.free.fr/
\href{http://mikl.crampon.free.fr/}{\Mundus}
http://perso.univ-rennes1.fr/ludovic.marquis
\href{http://perso.univ-rennes1.fr/ludovic.marquis}{\Mundus}
}

\begin{document}
\frontmatter





\begin{abstract}
On montre un lemme de Kazhdan-Margulis-Zassenhaus pour les géométries de Hilbert. Plus précisément, en toute dimension $n$, il existe une constante $\varepsilon_n > 0$ telle que, pour tout ouvert proprement convexe $\O$, pour tout point $x \in \O$, tout groupe discret engendré par un nombre fini d'automorphismes de $\O$ qui déplacent le point $x$ de moins de $\varepsilon_n$ est virtuellement nilpotent.
\end{abstract}

\begin{altabstract}
We prove a Kazhdan-Margulis-Zassenhaus lemma for Hilbert geometries. More precisely, in every dimension $n$ there exists a constant $\varepsilon_n > 0$ such that, for any properly convex open set $\O$ and any point $x \in \O$, any discrete group generated by a finite number of automorphisms of $\O$, which displace $x$ at a distance less than $\varepsilon_n$, is virtually nilpotent. 
\end{altabstract}

\maketitle


\mainmatter

\section*{Introduction}
Beaucoup d'attention a été portée jusqu'ici aux quotients compacts des géométries de Hilbert, plus connus sous le nom de \emph{convexes divisibles}. Suite aux premiers travaux de Victor Kac et \`Ernest Vinberg \cite{MR0208470}, Jacques Vey \cite{MR0283720} et Jean-Paul Benzécri \cite{MR0124005}, on peut citer entre autres ceux de William Goldman \cite{MR1053346}, Suhyoung Cho\"i \cite{MR1285533,MR1293655,MR2247648}, Choi-Goldman \cite{MR1145415,MR1414974}, et plus récemment d'Yves Benoist \cite{MR1767272,MR2010735,MR2094116,MR2195260,MR2218481,MR2295544} et Misha Kapovich \cite{MR2350468}.\\
En revanche, l'étude géométrique des quotients non compacts n'a été encore que peu abordée, si ce n'est dans les contributions récentes du second auteur et de Cho\"i. Dans \cite{Marquis:2009kq} et \cite{MR2740643}, ce sont les surfaces de volume fini qui sont à l'honneur; dans \cite{Marquis:2010fk}, des exemples de quotients de volume fini sont construits en toute dimension, suivant une idée de \cite{MR900823}; dans \cite{Choi:2010fk}, Cho\"i se donne une variété différentielle et étudie ses possibles structures projectives convexes. Dans l'article \cite{Crampon:2012fk} qui suivra celui-ci, nous introduirons et décrirons la notion de finitude géométrique pour les quotients de certaines géométries de Hilbert. Ainsi, notre point de vue sera opposé à celui de Cho\"i: nous partirons d'une géométrie de Hilbert $(\O,\d)$ donnée et il s'agira de comprendre la géométrie de ses éventuels quotients.\\
En fait, l'article présent ne devait être qu'une partie du suivant, mais nous avons décidé de l'écrire à part, d'abord pour l'intérêt intrinsèque de son résultat mais aussi pour raccourcir l'autre (ou l'inverse).\\\\

L'étude des géométries de Hilbert tire nombre de ses inspirations de la géométrie riemannienne hyperbolique ou de courbure négative. Le lemme de Kazhdan-Margulis-Zassenhaus est un résultat essentiel dans la compréhension des variétés riemanniennes de courbure négative ou nulle. Le but de cet article est de montrer un tel lemme pour les géométries de Hilbert:

\begin{theorem}\label{lem_mar}
En toute dimension $n$, il existe une constante $\varepsilon_n > 0$ et un entier $m_n$ tels que, pour tout ouvert proprement convexe $\O$, pour tout sous-groupe discret $\G$ de $\Aut(\O)$, pour tout point $x \in \O$, le groupe $\G_{\varepsilon_n}$ engendr\'e par l'ensemble $\{\g\in \G,\ d_{\O}(x,\g \cdot x) \leqslant \varepsilon_n\}$ est virtuellement nilpotent. De plus, on peut trouver un sous-groupe nilpotent de $\G_{\varepsilon_n}$ d'indice inférieur ou égal à $m_n$.
\end{theorem}

Le lemme de Kazhdan-Margulis-Zassenhaus pour les espaces symétriques a été démontré par David Kazhdan et Gregory Margulis dans \cite{MR0223487} en utilisant un résultat de Hans Julius Zassenhaus publié dans \cite{64.0961.06}, que nous utiliserons sous la forme du lemme \ref{lem_zas}. Il a été montré pour les espaces de Hadamard par Margulis; on pourra trouver dans \cite{MR0492072} quelques remarques, Margulis n'ayant pas publié de preuve complète. La page 150 du livre \cite{MR1253544} de Misha Gromov fournit un bref rappel historique de ce lemme.

Les lecteurs intéressés par une preuve complète peuvent consulter \cite{MR0507234} pour les espaces symétriques et \cite{MR823981} pour les variétés de Hadamard. 

Suhyoung Choï \cite{MR1405450} avait obtenu un lemme de Kazhdan-Margulis-Zassenhaus pour les géométries de Hilbert divisibles en dimension 2 et notre démonstration du théorème \ref{lem_mar} s'inspire d'une de ses idées. Une d\'emonstration ind\'ependante, bien qu'assez proche de la n\^otre, du th\'eor\`eme \ref{lem_mar} appara\^it \'egalement dans un travail r\'ecent de Daryl Cooper, Darren Long, and Stephan Tillmann \cite{Cooper:2011fk} sur les vari\'et\'es projectives convexes non compactes.\\

Pour terminer cette introduction, disons bri\`evement comment nous utiliserons ce r\'esultat dans \cite{Crampon:2012fk}. Soient $\O$ un ouvert proprement convexe de $\PP^n$, $\G$ un sous-groupe discret de $\Aut(\O)$ et $M=\Quo$ le quotient correspondant.\\
Fixons un r\'eel $\varepsilon > 0$ pour lequel le théorème \ref{lem_mar} s'applique. On note:
\begin{itemize}
 \item pour $x\in\O$, $\G_{\varepsilon}(x)$ le groupe engendré par les éléments $\g \in \G$ tels que $d_{\O}(x,\g \cdot x) < \varepsilon$ ;
 \item $\O_{\varepsilon} = \{  x \in \O \,\mid \, \G_{\varepsilon}(x) \textrm{ est infini} \}$;
 \item $M_{\varepsilon}=\Quotient{\O_{\varepsilon}}{\G}$ la projection de cet ensemble sur $M=\Quo$.

\end{itemize}
\par{
La partie $M_{\varepsilon}$ est la \emph{partie fine} de $M$. Dans le cas où $M$ est une variété, autrement dit quand $\G$ est sans torsion, c'est l'ouvert des points de $M$ dont le rayon d'injectivité est strictement inférieur à $\varepsilon/2$. Son complémentaire $M^{\varepsilon}$ est la \emph{partie épaisse} de $M$.
}
\par{
Le théorème \ref{lem_mar} permet de voir que le groupe fondamental de toute composante connexe de $M_{\varepsilon}$ est virtuellement nilpotent. Or, dans le cas o\`u $\O$ est strictement convexe et à bord $\C^1$, il n'est pas très difficile de classifier les sous-groupes discrets nilpotents de $\Aut(\O)$. Ce sont les premiers pas pour d\'ecrire la topologie et la g\'eom\'etrie de la partie fine.
}

\section{Géométrie de Hilbert}\label{geo_hilbert}
\par{
Cette partie constitue une introduction très rapide à la géométrie de
Hilbert. Pour une introduction beaucoup plus complète, on pourra lire \cite{MR2270228} ou le livre \cite{MR0075623}.
}


\par{
Une \emph{carte affine} $A$ de l'espace projectif réel $\PP^n(\R) = \PP^n$ est le complémentaire d'un hyperplan projectif. Une carte affine possède une structure naturelle d'espace affine. Un ouvert $\O$ de $\PP^n$ différent de $\PP^n$ est \emph{convexe} lorsqu'il est inclus dans une carte affine et qu'il est convexe dans cette carte. Un ouvert convexe $\O$ de $\PP^n$ est dit \emph{proprement convexe} lorsqu'il existe une carte affine contenant son adhérence $\overline{\O}$. Autrement dit, un ouvert convexe est proprement convexe lorsqu'il ne contient pas de droite affine. 
}
\\
\label{base}
\par{
Sur un ouvert proprement convexe $\O$ de $\PP^n$, la distance de Hilbert est définie de la façon suivante. Soient $x, y$ deux points distincts de $\O$; on note $a,b$ les points d'intersection de la droite $(xy)$ et du bord $\partial \O$ de $\O$ tels que $x$ soit entre $b$ et $y$ et $y$ entre $x$ et $a$ (voir figure \ref{dis}). La distance entre $x$ et $y$ est définie par la formule
}

$$
d_{\O}(x,y) = \frac{1}{2}\ln [a:b:x:y],
$$
où $[a:b:x:y]=\displaystyle\frac{\|a-x \|\cdot \|b-y\|}{\| a-y \| \cdot \| b-x \|}$ désigne le birapport des points $a,b,x,y$; $\| \cdot \|$ est une norme euclidienne quelconque sur une carte affine qui contient l'adhérence $\overline{\O}$ de $\O$.

\begin{center}
\begin{figure}[h!]
  \centering
\includegraphics[trim=0cm 0cm 0cm 12cm, clip=true, width=7cm]{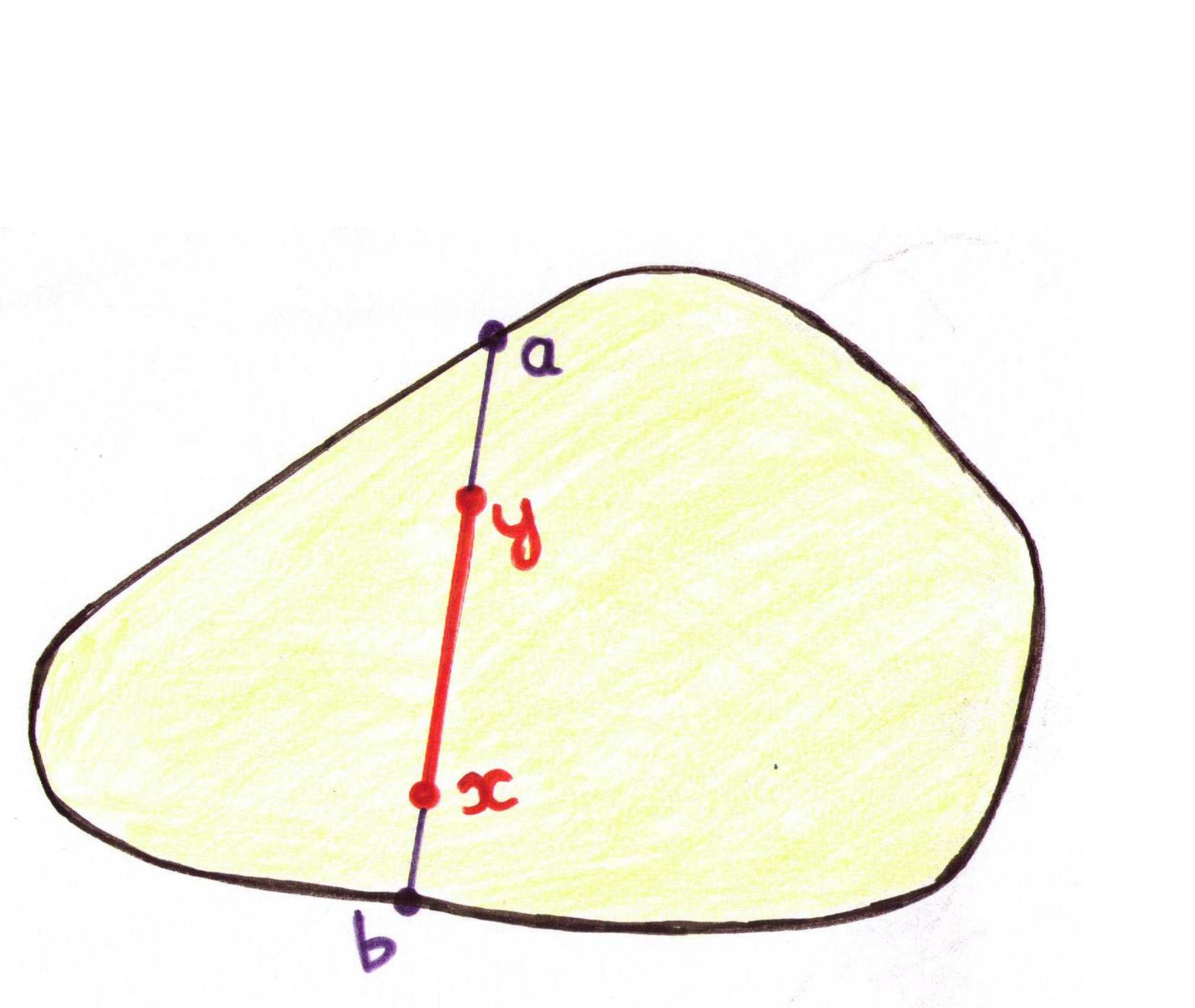}  
\caption{La distance de Hilbert
\label{dis}
}
\end{figure}
\end{center}

Il est clair que $d_{\O}$ ne dépend ni du choix de la carte, ni du choix de la norme euclidienne sur cette carte.\\ 

L'espace $(\O,d_{\O})$ ainsi construit est un espace métrique complet dont la topologie coïncide avec celle induite par $\PP^n$. Le groupe $\Aut(\O)$ des transformations projectives de $\ss$ qui préservent $\O$ est un sous-groupe fermé de $\ss$ qui agit par isométrie sur $(\O,d_{\O})$; il agit donc proprement sur $\O$.

%
%
%

\begin{rem}
\par{
Le groupe $\Isom(\O,\d)$ des isométries de l'espace métrique $(\O,d_{\O})$ n'est pas à priori r\'eduit \`a $\Aut(\O)$. C'est toutefois le cas si le convexe $\O$ est strictement convexe, ou un polytope, en \'ecartant le cas du simplexe pour lequel $\Aut(\O)$ est d'indice 2 dans le groupe $\Isom(\O,\d)$ (voir \cite{MR1238518} et \cite{MR2819195}). Dans d'autres cas bien particuliers, qui correspondent aux convexes homogènes, auto-duaux et non strictement convexes, le groupe $\Aut(\O)$ est aussi d'indice $2$ dans le groupe $\Isom(\O,\d)$ (voir \cite{Bosche:2012uq}). On pense qu'en général $\Aut(\O)$ est d'indice au plus $2$ dans $\Isom(\O,\d)$.
}
\\
\par{
Le th\'eor\`eme \ref{lem_mar} concerne le groupe $\Aut(\O)$. Notre méthode utilise le groupe $\ss$ et son action sur l'espace projectif $\PP^n$, et elle ne s'applique pas au groupe $\Isom(\O,\d)$. Le th\'eor\`eme \ref{lem_mar} reste toutefois vrai, comme conséquence d'un travail bien plus général d'Emmanuel Breuillard, Ben Green et Terence Tao \cite{pre06120992}. Ils montrent en effet que le lemme de Kazhdan-Margulis-Zassenhaus est vrai dans tout espace métrique qui satisfait la propriété suivante: il existe $K\in\N$ tel que toute boule de rayon $4$ est recouverte par au plus $K$ boules de rayon $1$. 
Voyons que les géométries de Hilbert ont cette propriété. Pour tout $R>0$, il existe des constantes $0<v(R)\leqslant V(R)$ telles que le volume de toute boule de rayon $R$ d'une quelconque géométrie de Hilbert est compris entre $v(R)$ et $V(R)$ (c'est une conséquence du théor\`eme de Benzécri \ref{theo_ben}; un résultat plus précis a été montré dans \cite{MR2245997}). Si $B$ est une boule de rayon $4$, le cardinal d'une partie $\{x_i\}$ $1/2$-séparée maximale de $B$ est alors majoré par un certain $K$ (indépendant de la géométrie et de la boule choisies), puisque les boules de rayon $1/4$ centrées en les $x_i$ sont disjointes. Par maximalité, les boules de rayon $1$ centrées en les $x_i$ recouvrent $B$.
}
\\
\par{
L'intérêt du travail présenté ici est bien de donner une preuve plus élémentaire que celle de \cite{pre06120992} pour le groupe $\Aut(\O)$.
}
\end{rem}

\color{black}
\section{Le théorème de Benzécri}

\par{
Pour montrer le théorème \ref{lem_mar}, nous aurons besoin d'un résultat de Benzécri ainsi que de certains lemmes intermédiaires de sa démonstration, que nous rappelons ici. Toutes les démonstrations de ces résultats se trouvent dans les notes de cours de Goldman \cite{NoteGoldman}.
}
\\
\par{
On définit l'espace $\E$ des convexes marqués comme l'ensemble 
$$\E = \{  (\O,x) \,\mid\, \O \textrm{ est un ouvert proprement convexe de } \PP^n \textrm{ et } x \in \O \},$$

muni de la topologie de Hausdorff héritée de la distance canonique sur $\PP^n$.}

\begin{theo}[Benzécri \cite{MR0124005}]\label{theo_ben}
L'action de $\ss$ sur $\E$ est propre et cocompacte.
\end{theo}

\par{
Un bon exemple  d'ouvert proprement convexe est le suivant. Soit 
$$q(x) = x_1^2+...+x_n^2-x_{n+1}^2$$ 
une forme quadratique de signature $(n,1)$ sur $\R^{n+1}$. L'ensemble $\C = \{ x \in \R^{n+1} \,\mid\, q(x) < 0 \}$ est un cône ouvert convexe de $\R^{n+1}$ et la projection de $\C$ sur $\PP^n$ sera noté $\E_0$. Un \emph{ellipsoïde} de $\PP^n$ est l'image de $\E_0$ par un élément de $\ss$.
}
\\
\par{
La proposition suivante définit la notion d'ellipsoïde d'inertie d'un convexe borné d'un \textsl{espace affine} (si l'énoncé nécessite une métrique euclidienne ``accessoire'', l'ellipsoïde d'inertie ne dépend pas de cette métrique).
}

\begin{prop}
Soit $\O$ un convexe borné de l'espace $\R^n$ euclidien, de centre de gravité $O$. Alors il existe un unique ellipsoïde $\E$, appelé \emph{ellipsoïde d'inertie de} $\O$, tel que, pour toute application affine $\psi: \R^n \rightarrow \R$ telle que $\psi(O)=0$, les moments d'inertie de $\O$ et $\E$ par rapport à $O$ de l'application affine $\psi$ sont égaux, c'est à dire que l'on a l'égalité suivante: $$\int_{\O} \psi^2(x)\ dx = \int_{\E} \psi^2(x)\ dx.$$
\end{prop}

\par{
Cette proposition est très classique en mécanique du solide. Le lecteur intéressé pourra trouver une démonstration, des références, et une description beaucoup plus complète de l'ellipsoïde d'inertie dans \cite{MR1008717}. Nous nous contenterons d'utiliser l'existence d'un tel ellipsoïde et la proposition \ref{etap_ben_2} ci-dessous, que l'on trouve dans \cite{NoteGoldman}. Les deux propositions qui suivent peuvent être vues comme les étapes principales qui mènent au théorème de Benzécri.
}

\begin{prop}[\cite{NoteGoldman}]\label{etap_ben_1}
Soit $(\O,x) \in \E$. Il existe une unique carte affine de $\PP^n$ dans laquelle l'ouvert $\O$ est borné et le point $x$ est le centre de gravité de $\O$.
\end{prop}

\begin{prop}[\cite{NoteGoldman}]\label{etap_ben_2}
Pour $n\geqslant 1$, il existe des constantes $R_n > r_n > 0$ telles que, pour tout ouvert convexe borné $\O$ de l'espace $\R^n$ euclidien dont le centre de gravité est l'origine $O$ et l'ellipsoïde d'inertie la boule unité, on ait 
$$B_{r_n}(O) \subset \O \subset B_{R_n}(O),$$
où $B_r(O)$ représente la boule de centre $O$ et de rayon $r\geqslant 0$.
\end{prop}

\begin{defi}
On appellera la carte affine définie par l'hyperplan $\{ x_{n+1} = 1 \}$ \emph{la carte standard} de $\PP^n$. On munit cette carte du produit scalaire euclidien canonique. On note $O$ le point de coordonnée $[0:0:\cdots : 1]$ que l'on choisit comme origine de cette carte.

\noindent On dira qu'un couple $(\O,x)$ est \emph{standard} lorsque:
\begin{enumerate}
\item le convexe $\O$ est borné dans la carte standard de $\PP^n$;
\item le point $x$ est le centre de gravité de $\O$ dans la carte standard;
\item le point $x$ est l'origine $O$ de la carte standard;
\item l'ellipsoïde d'inertie de $\O$ est la boule unité centrée en $x$ dans la carte standard.
\end{enumerate}
On note $\E_s$ l'ensemble des couples standards de $\E$.
\end{defi}

\begin{rema}\label{rem_derniere_minute}
On remarquera que si $(\O,x)$ est un couple standard et si l'élément $\g \in \ss$ est un automorphisme $\O$ qui fixe $x$, alors $\g$ préserve l'ellipsoïde d'inertie de $\O$, c'est-à-dire la boule unité de la carte standard. Par conséquent, les stabilisateurs $\Stab_{\Aut(\O)}(x)$ de $x$ dans $\Aut(\O)$, pour $(\O,x)\in\E_s$, sont \underline{tous} inclus dans le groupe $\textrm{SO}_n(\R)$.
\end{rema}

\begin{prop}\label{prop_stand}
L'ensemble $\E_s$ est compact et pour tout couple $(\O,x) \in \E$, il existe un élément $g \in \ss$ tel que $g \cdot (\O,x) \in \E_s$.
\end{prop}

\begin{proof}
Une limite d'ensembles convexes pour la topologie de Hausdorff est encore un ensemble convexe. La proposition \ref{etap_ben_2} montre que l'ensemble $\E_s$ s'identifie à un ensemble $H$ d'ouverts convexes coincés entre les boules $B_{r_n}(O)$ et $B_{R_n}(O)$. Par suite, toute partie de $\PP^n$ adhérente à $H$ est convexe, proprement convexe et l'ensemble $\E_s$ est donc fermé et finalement compact. La proposition \ref{etap_ben_1} montre la seconde partie de cette proposition.
\end{proof}


%
%
%
%
%

\begin{rema}
Le deuxième point de la proposition \ref{prop_stand} montre qu'il suffit de prouver le théorème \ref{lem_mar} pour tout couple standard et c'est ce que nous ferons.
\end{rema}

\section{Démonstration du théorème \ref{lem_mar}}

\subsection{Le lemme de Zassenhaus}
\par{
Comme annoncé, nous allons avoir besoin du lemme de Zassenhaus.
}

\begin{lemm}[Zassenhaus \cite{64.0961.06} ou le théorème 8.16 de \cite{MR0507234}]\label{lem_zas}
Soit $G$ un groupe de Lie. Il existe un voisinage de l'identité $\,\U$ tel que tout groupe discret engendré par des éléments de $\U$ est nilpotent.
\end{lemm}

\subsection{Trois lemmes techniques de théorie des groupes}

\begin{lemm}\label{lem_entre}
Soient $G$ un groupe de Lie, $K$ un sous-groupe compact de $G$ et $m\in\N\smallsetminus\{0\}$. Pour tout voisinage ouvert relativement compact $V$ de $K$, il existe un voisinage ouvert relativement compact $W$ de $K$ tel que:
\begin{enumerate}
\item $W=W^{-1}$;
\item $W^m \subset V$.
\end{enumerate}
\end{lemm}

\begin{proof}
C'est une simple conséquence du fait que la multiplication est continue.
\end{proof}

\begin{lemm}\label{lem_groupe_compact}
Soient $G$ un groupe compact métrisable et $\V$ un voisinage ouvert de l'identité dans $G$. Il existe $m\in\N\smallsetminus\{0\}$ tel que tout sous-groupe fermé $K$ de $G$ est inclus dans la réunion de $m$ translatés de $\V$ par des éléments de $K$.
\end{lemm}

\begin{proof}
Tout sous-groupe fermé $K$ de $G$ est compact et recouvert par $\bigcup_{k\in K} k\V$, puisque $\V$ contient l'identité; tout sous-groupe fermé $K$ de $G$ est donc recouvert par un nombre fini de translatés de $\V$ par des éléments de $K$.\\
Supposons maintenant, en raisonnant par l'absurde, qu'il existe une suite $(K_m)_{m\in\N}$ de sous-groupes fermés de $G$ telle que toute réunion de $m$ translatés de $\V$ par des éléments de $K_m$ ne contienne pas tout $K_m$. 

La distance de Hausdorff fait de l'ensemble des fermés de $G$ un espace compact, la topologie induite sur l'ensemble fermé des sous-groupes fermés de $G$ en fait un ensemble compact. On peut donc supposer que la suite $(K_m)_{m\in\N}$ converge pour cette topologie vers un sous-groupe fermé $K$ de $G$. Ce groupe $K$ n'est, par définition de la suite $(K_m)_{m\in\N}$, inclus dans aucune réunion finie de translatés de $\V$ par des éléments de $K$. D'où une contradiction.
\end{proof}

Nous aurons besoin du lemme suivant pour ``gérer l'indice fini''.

\begin{lemm}\label{lem_ber}
Soient $m\in\N\smallsetminus\{0\}$, $\G$ un groupe agissant transitivement sur un ensemble fini $E$, de 
cardinal $|E|>m$, et $S$ un ensemble de générateurs de $\G$, tel que si $g \in S$, alors $g^{-1} \in S$. Si $e$ est un élément de $E$, il existe des éléments $\g_1,...,\g_{m+1} \in S^m$ qui envoient $e$ sur $m + 1$ points distincts de $E$. 
\end{lemm}

\begin{proof}
On pose $S_*=S\cup \{ 1 \}$. Pour $n\in\N$, on note $N_n$ le cardinal de la partie $S_*^n \cdot e$ de l'orbite $\G \cdot e = E$; l'énoncé revient à montrer que $N_m>m$. Or, la suite $(N_n)$ est strictement croissante jusqu'à un certain rang $i\in \N$ pour lequel elle devient stationnaire: pour $n\geqslant i$, $N_n=N_i$. Comme $\G$ agit transitivement sur $E$, on a forcément $N_i=|E|>m$. Ainsi, si $i\leqslant m$, alors $N_m>m$; sinon, le fait que $N_0=1$ et que $(N_n)$ est strictement croissante entre $0$ et $m$ implique aussi $N_m>m$. 
\end{proof}

\subsection{Lemmes clés de géométrie de Hilbert}

On munit le groupe $\ss$ d'une distance $d_{SL}$ invariante à gauche et propre, c'est-à-dire que les boules fermées sont compactes.

\begin{lemm}\label{lem_mar1}
Pour tout $\varepsilon > 0$, il existe $\delta > 0$ tel que, pour tout couple standard $(\O,x) \in \E_s$ et tout élément $\g \in \Aut(\O)$, on ait

$$d_{\O}(x, \g \cdot x) \leqslant \varepsilon\ \Longrightarrow\ d_{\sss}(1,\g) \leqslant \delta.$$
\end{lemm}

\begin{proof} 
On note $x(\O)$ le centre de gravité d'un convexe $\O$ borné dans la carte standard. Pour tout $\varepsilon >0$, l'ensemble 
$$K=\{(\O,x)\in \E\ |\ \O\ \textrm{borné dans la carte standard et}\ d_{\O}(x, x(\O)) \leqslant \varepsilon\},$$ 
est un voisinage compact de $\E_s$. Le théorème \ref{theo_ben} affirme que l'action de $\ss$ sur $\E$ est propre. Par conséquent, l'ensemble des éléments de $\ss$ qui envoient un élément de $\E_s$ sur un élément de $K$ est compact, donc borné. 
\end{proof}

Si $x$ est un point d'un ouvert proprement convexe $\O$ de $\PP^n$, on notera $\Stab_{\O}(x)$ le stabilisateur de $x$ dans $\Aut(\O)$.

\begin{lemm}\label{lem_mar2}
Pour tout $\varepsilon > 0$, il existe $\delta > 0$ tel que, pour tout couple standard $(\O,x) \in \E_s$, pour tout $\g \in \Aut(\O)$, on ait
$$d_{\O}(x, \g \cdot x) \leqslant \delta\ \Longrightarrow\ d_{\sss}(\Stab_{\O}(x),\g) \leqslant \varepsilon.$$
\end{lemm}

\begin{proof}
On démontre ce lemme par l'absurde. Supposons qu'il existe une constante $\varepsilon_0 >0$, une suite de couples standards $(\O_n,x_n) \in \E_s$ et une suite d'éléments $\g_n \in \Aut(\O_n)$ tels que

\begin{itemize}
\item $d_{\O_n}(x_n, \g_n \cdot x_n) \underset{n \to \infty}{\longrightarrow} 0$;
\item $d_{\sss}(\Stab_{\O_n}(x_n),\g_n) \geqslant \varepsilon_0$.
\end{itemize}

L'espace $\E_s$ est compact. Le lemme \ref{lem_mar1} montre que la suite $(\g_n)_{n \in \N}$ reste dans un compact de $\ss$. On peut donc supposer que toutes ces suites convergent. Les limites, notées sans indice, vérifient les points suivants:

\begin{itemize}
\item $(\O,x) \in \E_s$;
\item $\g \in \Aut(\O)$;
\item $\g\cdot x = x$;
\item $d_{\sss}(\Stab_{\O}(x),\g) \geqslant \varepsilon_0$.
\end{itemize}
Les trois premiers points montrent que $\g \in \Stab_{\O}(x)$ et le dernier montre que c'est absurde.
\end{proof}

\subsection{Démonstration du théorème \ref{lem_mar}}

On peut maintenant donner une

\begin{proof}[Démonstration du théorème \ref{lem_mar}]
Le lemme \ref{lem_zas} montre qu'il existe un voisinage ouvert $\U$ de l'identité dans $\ss$ tel que tout groupe discret engendré par des éléments de $\U$ est nilpotent. 

Pour tout couple standard $(\O,x) \in \E_s$, le stabilisateur $K$ de $x$ dans $\Aut(\O)$ est un groupe compact inclus dans le groupe compact $G= \textrm{SO}_n(\R)$ (remarque \ref{rem_derniere_minute}). Le lemme \ref{lem_groupe_compact} montre alors qu'il existe $m\in\N\smallsetminus\{0\}$ tel que $K$ est recouvert par $m$ translatés de $\U$ par $K$; de plus, $m$ peut être choisi indépendamment du couple standard $(\O,x)$. 

On fixe maintenant un couple standard $(\O,x) \in \E_s$ et on note $K$ le stabilisateur de $x$ dans $\Aut(\O)$. On note
\begin{equation}\label{reunionv}
V = \bigcup_{i=1}^m k_i \U
\end{equation}  

\noindent la réunion recouvrant $K$ des translatés de $\U$ par les éléments $k_1,\cdots,k_m\in K$; $V$ est un voisinage ouvert relativement compact de $K$.

Pour $\eta > 0$, on note $W_{(\O,x,\eta)} = \{\g \in \ss \mid d_{SL}(\Stab_{\O}(x),\g) \leqslant \eta\}$ le voisinage compact de $K$ de taille $\eta$. Le lemme \ref{lem_entre} montre que pour tout couple $(\O,x) \in \E_s$, il existe un nombre $\eta_{(\O,x)}>0$ tel que $W_{(\O,x,\eta_{(\O,x)})}^{m} \subset V$. L'espace $\E_s$ est compact et un raisonnement par l'absurde montre que l'on peut choisir $\eta_{(\O,x)}$ indépendamment du couple $(\O,x)$. Autrement dit, il existe un $\eta>0$ tel que $W_{(\O,x,\eta)}^m\subset V$ pour tout couple standard $(\O,x)\in\E_s$. Par la suite, on abrègera la notation en $W=W_{(\O,x,\eta)}$.\\

Le lemme \ref{lem_mar2} montre qu'il existe $\varepsilon_n >0$ tel que, pour tout couple $(\O,x) \in \E_s$, l'ensemble $F_{\varepsilon_n}(x) = \{ \g \in \Aut(\O) \,\mid\,  d_{\O}(x, \g \cdot x) \leqslant \varepsilon_n \}$ est inclus dans $W$. 
On va montrer qu'un groupe discret $\G$ engendré par des éléments de $F_{\varepsilon_n}(x)$ est virtuellement nilpotent.\\

Le lemme \ref{lem_zas} montre que le groupe $\Lambda$ engendré par $\G \cap \U$ est nilpotent. On va voir que $\Lambda$ est d'indice fini dans $\G$ inférieur à $m$.  

Le groupe $\G$ agit transitivement sur l'ensemble $E=\Quotient{\G}{\Lambda}$. Supposons que cet ensemble contienne strictement plus de $m$ éléments. Le lemme \ref{lem_ber} montre qu'il existe des éléments $\g_1,...,\g_{m+1} \in \G$, produits d'au plus $m$ éléments de $F_{\varepsilon_n}(x)$, qui envoient $\Lambda$ sur $m+1$ classes à gauche distinctes de $E$. 

Comme $F_{\varepsilon_n}(x) \subset W$ et $W^m \subset V$, les éléments $\g_1,...,\g_{m+1}$ sont dans $V$. Or, le voisinage $V$ est recouvert par $m$ translatés de $\U$ via des éléments de $K$; il existe donc deux éléments distincts parmi $\g_1,...,\g_{m+1}$ qui  appartiennent à un même translaté de $\U$ dans la réunion (\ref{reunionv}). On peut supposer que ce sont les éléments $\g_1$ et $\g_2$ qui sont dans $k_1 \U$; ainsi, on peut écrire $\g_1 = k_1 u_1$, $\g_2=k_1 u_2$ avec $u_1, u_2 \in \G \cap \U$. Il vient que $\g_1^{-1} \g_2  = u_1^{-1} k_1^{-1} k_1 u_2$ et donc $\g_1^{-1} \g_2 \in \Lambda$. On en déduit que les éléments $\g_1$ et $\g_2$ envoient $\Lambda$ sur la même classe à gauche de $E$, d'où une contradiction.  
\end{proof}

\newpage

\bibliographystyle{alpha}

\end{document}